\documentclass[12pt]{article}
\usepackage[amsmath]{e-jc}

\usepackage{graphicx}
\usepackage{xurl}
\usepackage{tikz}
\usetikzlibrary{tikzmark,decorations.pathreplacing}
\usetikzlibrary{arrows,matrix}
\usetikzlibrary{positioning}
\usetikzlibrary{decorations}

\dateline{Oct 29, 2021}{Nov 17, 2022}{TBD}


\Copyright{The author. Released under the CC BY license (International 4.0).}

\title{Counting Baxter Matrices}

\author{George Spahn\\
\small Department of Mathematics\\[-0.8ex]
\small Rutgers University (New Brunswick)\\[-0.8ex] 
\small New Jersey, U.S.A.\\
\small\tt www.georgespahn.com}

\begin{document}

\maketitle


\begin{abstract}
  Donald Knuth recently introduced the notion of a Baxter matrix, generalizing Baxter permutations. We show that for fixed number of rows, $r$, the number of Baxter matrices with $r$ rows and $k$ columns eventually satisfies a polynomial in $k$ of degree $2r-2$. We also give a proof of Knuth's conjecture that the number of 1s in an $r \times k$ Baxter matrix is less than $r+k$.
\end{abstract}

\section{Introduction}

Baxter permutations are a class of permutations that come up in numerous places. They were introduced in 1964 by Glen Baxter, who was studying fixed points of commuting continuous functions \cite{Baxter}. They have since been connected to Hopf Algebras \cite{GS}, planar graphs \cite{Bonichon}, and tilings \cite{Korn}. The number of Baxter permutations of length $n$ is given by A001181 in the OEIS and forms a holonomic sequence of order 2 and degree 2 \cite{Oeis}. More recently, Donald Knuth suggested a generalization of Baxter permutations to matrices \cite{Knuth}. 

Donald Knuth defines a Baxter matrix as a matrix of 0s and 1s that satisfy 4 conditions. Before we state the conditions, we must define a pinwheel in a matrix.

\begin{definition} 
  Let $M$ be an $r \times k$ matrix. For $1 \leq x \leq r-1$ and $1 \leq y \leq k-1$, the \textbf{clockwise pinwheel} of index $(x,y)$, denoted $P_{x,y}$, is a specific subset of the entries of $M$.  $P_{x,y}$ is divided into 4 \textbf{segments}. 

\begin{itemize}
    \item Segment $A_{x,y}$ contains the entries $M[i,y+1]$, for $1 \leq i \leq x$.
    \item Segment $B_{x,y}$ contains the entries $M[x,i]$, for $1 \leq i \leq y$.
    \item Segment $C_{x,y}$ contains the entries $M[x+1,i]$, for $y+1 \leq i \leq k$.
    \item Segment $D_{x,y}$ contains the entries $M[i,y]$, for $x+1 \leq i \leq r$.
\end{itemize}

Similarly, for $1 \leq x \leq r-1$ and $1 \leq y \leq k-1$, the \textbf{counterclockwise pinwheel} of index $(x,y)$, denoted $P'_{x,y}$ also contains four segments.
\begin{itemize}
    \item Segment $A'_{x,y}$ contains the entries $M[i,y]$, for $1 \leq i \leq x$.
    \item Segment $B'_{x,y}$ contains the entries $M[x+1,i]$, for $1 \leq i \leq y$.
    \item Segment $C'_{x,y}$ contains the entries $M[x,i]$, for $y+1 \leq i \leq k$.
    \item Segment $D'_{x,y}$ contains the entries $M[i,y+1]$, for $x+1 \leq i \leq r$.
\end{itemize}
\end{definition}

\begin{definition}
    A \textbf{Baxter Matrix} is a matrix of 0s and 1s that satisfy 4 conditions.

\begin{enumerate}
    \item Each row contains at least one 1.
    \item Each column contains at least one 1.
    \item For each clockwise pinwheel, at least one of the four segments has only 0s.
    \item For each counterclockwise pinwheel, at least one of the four segments has only 0s.
\end{enumerate}
\end{definition}

A pinwheel is said to be satisfied if at least one of its four segments has only 0s. In an $r \times k$ matrix, there are $(r-1)(k-1)$ pinwheels of each direction that must be satisfied. As an example consider the two matrices in Figure \ref{fig:examples}. The left one is indeed a Baxter matrix, and the reader is encouraged to check that all the pinwheels are satisfied. The right one is almost a Baxter matrix; every pinwheel except $P'_{2,2}$ is satisfied. Note that in general a Baxter matrix could have multiple 1s in a row or column.

\begin{figure}[h]
    \centering
   \begin{tikzpicture}
        \matrix [matrix of math nodes,left delimiter=(,right delimiter=)] 
        {
            0 &1 &0 &0 \\               
            1 &0 &0 &0 \\               
            0 &0 &0 &1 \\  
            0 &0 &1 &0 \\           
        };  
        \begin{scope}[xshift=3cm]
\matrix [matrix of math nodes,left delimiter=(,right delimiter=)] 
        {
            0 &1 &0 &0 \\               
            0 &0 &0 &1 \\               
            1 &0 &0 &0 \\  
            0 &0 &1 &0 \\           
        };  
\end{scope}
\begin{scope}[xshift=6cm]
        \matrix [matrix of math nodes,left delimiter=(,right delimiter=)] (m)
        {
            0 &1 &0 &0 \\               
            0 &0 &0 &1 \\               
            1 &0 &0 &0 \\  
            0 &0 &1 &0 \\           
        };  
        \draw[color=black] (m-3-1.north west) -- (m-3-2.north east) -- (m-3-2.south east) -- (m-3-1.south west);
        \draw[color=black] (m-1-2.north west) -- (m-2-2.south west) -- (m-2-2.south east) -- (m-1-2.north east);
        \draw[color=black] (m-2-4.north east) -- (m-2-3.north west) -- (m-2-3.south west) -- (m-2-4.south east);
        \draw[color=black] (m-4-3.south west) -- (m-3-3.north west) -- (m-3-3.north east) -- (m-4-3.south east);  
\end{scope}   
    \end{tikzpicture}
    \caption{A Baxter matrix (left), and a non-Baxter matrix (middle), with its unsatisfied pinwheel (right).}
     \label{fig:examples}
\end{figure}

It turns out that $r \times r$ Baxter matrices that only have a single 1 in each row and column are in bijection with Baxter permutations of length $r$. This can be seen by viewing the matrix as a permutation matrix. For each column $i$, the corresponding permutation $\pi$ maps $i$ to $r$ minus the row in which the 1 appears in column $i$.

What is the maximum number of 1s that can appear in a Baxter matrix of size $r \times k$? Knuth conjectured that this maximum is equal to $r+k-1$, and exhaustively checked it for all $r$ and $k$ $\leq 7$. In this paper we will prove that the conjecture is true, i.e. that for any $r,k \geq 1$, the number of 1s in an $r \times k$ Baxter matrix is less than $r+k$.
\section{A  Finite State Automaton for Baxter Matrices with r rows}
In this section we will describe a finite state automaton for determining whether a matrix is Baxter. We will fix the number of rows, and have the automaton read the columns of the matrix as symbols. When it is done reading the columns, it should accept or reject according to whether the matrix is Baxter. For ease of explanation, consider first the $r=2$ case.

\subsection{The 2 Row Case}

Let's construct a finite state machine for determining whether a 0-1 matrix with 2 rows is a Baxter Matrix. The symbols that our machine should recognize as input should be the possible columns in the matrix. There are 4 possible columns in a 0-1 matrix with 2 rows, [0,0]$^T$, [0,1]$^T$, [1,0]$^T$, and [1,1]$^T$. We can ignore the column [0,0]$^T$ because any column in a baxter matrix must not be all zeros. 

As we move through the columns of our input matrix, our machine will keep track of the following information for each row: whether the row is all 0s up to this point (so that it can be used to satisfy future pinwheels) and whether the row must be all 0s in the future (because a pinwheel from earlier is depending on it). Thus each row can be in one of 4 possible states; we will refer to them as the 4 rowstates: 
\begin{enumerate}
    \item This row has a 1 in the most recent column.
    \item This row only contains 0s up to now.
    \item This row must only contain 0s for the rest of the columns, including the most recent one.
    \item This row had a 0 in the most recent column but does not fit 2. or 3.
\end{enumerate}
Our machine will have 16 states, one for each ordered pair of rowstates. We next remove the states which have all rows in rowstates 2, 3, and 4. This is because if all rows had a 0 most recently, then the corresponding column contains only 0s, and the matrix cannot be Baxter. This leaves us with $16 - 9 = 7$ states.  I claim that the information contained in such a state is enough to determine which columns can come next. Suppose we just read column $m$, and have a proposed column $m+1$. For both pinwheels of index $(1,m)$ we can mostly check whether they are satisfied with the information stored. The information about whether the vertical strips are all 0s is known because we store the exact contents of the most recent column in the state. The horizontal strip going left is available if and only if the corresponding row is in rowstate 2. The only thing we don't know yet is whether a horizontal strip going off to the right is all 0s, but if a pinwheel requires it to be so, we can set the rowstate of the corresponding row to 3, and keep track of it for later. We note some other basic constraints:
\begin{itemize}
    \item Once a row is in rowstate 3 it can never leave rowstate 3.
    \item Once a row leaves rowstate 2 it can never come back to rowstate 2.
    \item A row cannot transition from 2 to 4.
    \item A row cannot transition from 2 to 3 directly, or else it will be all zeros.
\end{itemize}
As the automaton proceeds reading columns, it checks whether each new pinwheel can be satisfied. When it encounters such a pinwheel that cannot be satisfied, it can immediately reject the sequence of columns. If a pinwheel cannot be satisfied given the first $j$ columns, there is no Baxter matrix that begins with those first $j$ columns. Similarly, if a pinwheel is satisfied given the first $j$ columns, we do not need to keep track of that pinwheel any longer. It will still be satisfied in any Baxter matrix with those first $j$ columns provided we keep track of which rows must be all 0s in the future.

We additionally add a start state that transitions to all states that have each row in either rowstate 1 or rowstate 2 (rowstates 3 and 4 cannot be reached using only a single column). We designate all states that have no rows in rowstate 2 as accept states (2 must be excluded so that no row of the final matrix is all 0s). Enforcing all of the rules we have described so far yields the automaton in Figure \ref{fig:a2}. The state label 12 indicates that the first row is in rowstate 1 and the second row is in rowstate 2. The transition label 10 indicates the column [1,0]$^T$. \\

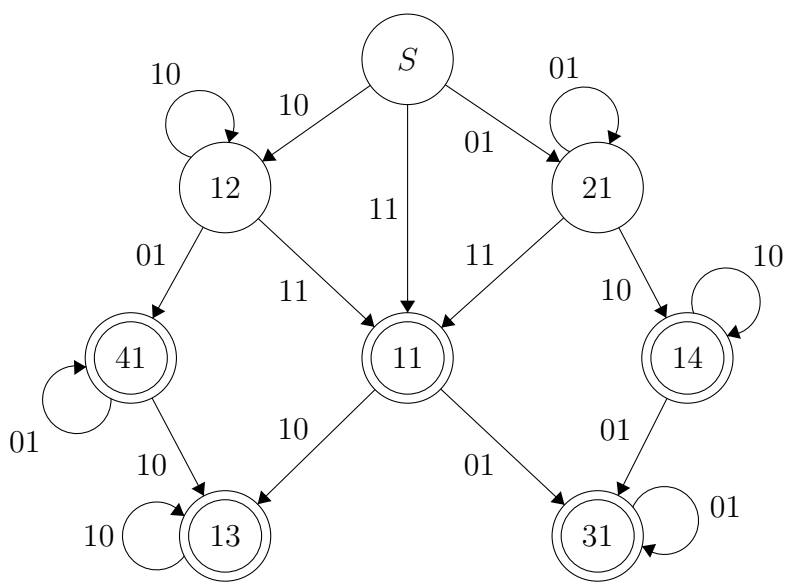
\begin{figure}[h]
    \centering
    \begin{tikzpicture}[scale=0.2]
\tikzstyle{every node}+=[inner sep=0pt]
\draw [black] (26.3,-3.2) circle (3);
\draw (26.3,-3.2) node {$S$};
\draw [black] (14.3,-11.7) circle (3);
\draw (14.3,-11.7) node {$12$};
\draw [black] (38.8,-11.7) circle (3);
\draw (38.8,-11.7) node {$21$};
\draw [black] (26.3,-23) circle (3);
\draw (26.3,-23) node {$11$};
\draw [black] (26.3,-23) circle (2.4);
\draw [black] (44.7,-23) circle (3);
\draw (44.7,-23) node {$14$};
\draw [black] (44.7,-23) circle (2.4);
\draw [black] (8.1,-23) circle (3);
\draw (8.1,-23) node {$41$};
\draw [black] (8.1,-23) circle (2.4);
\draw [black] (14.3,-34.8) circle (3);
\draw (14.3,-34.8) node {$13$};
\draw [black] (14.3,-34.8) circle (2.4);
\draw [black] (38.8,-34.8) circle (3);
\draw (38.8,-34.8) node {$31$};
\draw [black] (38.8,-34.8) circle (2.4);
\draw [black] (23.85,-4.93) -- (16.75,-9.97);
\fill [black] (16.75,-9.97) -- (17.69,-9.91) -- (17.11,-9.1);
\draw (18.8,-6.95) node [above] {$10$};
\draw [black] (26.3,-6.2) -- (26.3,-20);
\fill [black] (26.3,-20) -- (26.8,-19.2) -- (25.8,-19.2);
\draw (25.8,-13.1) node [left] {$11$};
\draw [black] (28.78,-4.89) -- (36.32,-10.01);
\fill [black] (36.32,-10.01) -- (35.94,-9.15) -- (35.38,-9.98);
\draw (31.05,-7.95) node [below] {$01$};
\draw [black] (16.48,-13.76) -- (24.12,-20.94);
\fill [black] (24.12,-20.94) -- (23.88,-20.03) -- (23.19,-20.76);
\draw (18.82,-17.83) node [below] {$11$};
\draw [black] (36.57,-13.71) -- (28.53,-20.99);
\fill [black] (28.53,-20.99) -- (29.45,-20.82) -- (28.78,-20.08);
\draw (31.07,-16.86) node [above] {$11$};
\draw [black] (24.16,-25.1) -- (16.44,-32.7);
\fill [black] (16.44,-32.7) -- (17.36,-32.49) -- (16.66,-31.78);
\draw (18.78,-28.42) node [above] {$10$};
\draw [black] (28.48,-25.06) -- (36.62,-32.74);
\fill [black] (36.62,-32.74) -- (36.38,-31.83) -- (35.69,-32.56);
\draw (31.03,-29.38) node [below] {$01$};
\draw [black] (12.083,-9.697) arc (255.63516:-32.36484:2.25);
\draw (10.38,-4.9) node [above] {$10$};
\fill [black] (14.54,-8.72) -- (15.22,-8.07) -- (14.26,-7.82);
\draw [black] (12.86,-14.33) -- (9.54,-20.37);
\fill [black] (9.54,-20.37) -- (10.37,-19.91) -- (9.49,-19.43);
\draw (10.53,-16.16) node [left] {$01$};
\draw [black] (36.978,-9.332) arc (245.30993:-42.69007:2.25);
\draw (36.63,-4.45) node [above] {$01$};
\fill [black] (39.57,-8.81) -- (40.36,-8.29) -- (39.45,-7.88);
\draw [black] (40.19,-14.36) -- (43.31,-20.34);
\fill [black] (43.31,-20.34) -- (43.38,-19.4) -- (42.5,-19.86);
\draw (41.07,-18.5) node [left] {$10$};
\draw [black] (45.117,-20.041) arc (199.71312:-88.28688:2.25);
\draw (50,-16.92) node [above] {$10$};
\fill [black] (47.3,-21.53) -- (48.22,-21.73) -- (47.89,-20.79);
\draw [black] (6.797,-25.689) arc (1.87498:-286.12502:2.25);
\draw (2.2,-28.56) node [left] {$01$};
\fill [black] (5.17,-23.6) -- (4.36,-23.13) -- (4.39,-24.13);
\draw [black] (9.5,-25.66) -- (12.9,-32.14);
\fill [black] (12.9,-32.14) -- (12.98,-31.2) -- (12.09,-31.67);
\draw (10.52,-30.05) node [left] {$10$};
\draw [black] (11.62,-36.123) arc (-36:-324:2.25);
\draw (7.05,-34.8) node [left] {$10$};
\fill [black] (11.62,-33.48) -- (11.27,-32.6) -- (10.68,-33.41);
\draw [black] (41.109,-32.903) arc (157.12921:-130.87079:2.25);
\draw (46.13,-32.86) node [right] {$01$};
\fill [black] (41.71,-35.48) -- (42.25,-36.25) -- (42.64,-35.33);
\draw [black] (43.36,-25.68) -- (40.14,-32.12);
\fill [black] (40.14,-32.12) -- (40.95,-31.62) -- (40.05,-31.18);
\draw (41.05,-27.79) node [left] {$01$};
\end{tikzpicture}\\
    \caption{The automaton for the 2 row case.}
    \label{fig:a2}
\end{figure}

We can do this same process for any fixed number of rows. There will be $2^r$ symbols, and $4^r - 3^r$ states. Let's call this automaton $A_r$. Figure \ref{fig:a3} contains $A_3$:

\begin{figure}[h]
    \centering
    \begin{tikzpicture}[scale=0.2]
\tikzstyle{every node}+=[inner sep=0pt]

\draw [black] (37.67,-34.1) -- (36.43,-53.7);
\fill [black] (36.43,-53.7) -- (36.98,-52.94) -- (35.98,-52.87);
\draw [black] (36.81,-33.84) -- (25.39,-53.96);
\fill [black] (25.39,-53.96) -- (26.22,-53.51) -- (25.35,-53.02);
\draw [black] (38.43,-34) -- (45.07,-53.8);
\fill [black] (45.07,-53.8) -- (45.29,-52.89) -- (44.34,-53.2);
\draw [black] (18.08,-22.46) -- (12.72,-30.44);
\fill [black] (12.72,-30.44) -- (13.58,-30.06) -- (12.75,-29.5);
\draw [black] (17.58,-21.97) -- (5.22,-30.93);
\fill [black] (5.22,-30.93) -- (6.16,-30.86) -- (5.57,-30.05);
\draw [black] (18.18,-22.52) --  (6.12,-42.78);
\fill [black] (6.12,-42.78) -- (6.96,-42.35) -- (6.1,-41.84);
\draw [black] (18.75,-22.75) -- (14.15,-42.55);
\fill [black] (14.15,-42.55) -- (14.82,-41.89) -- (13.85,-41.66);
\draw [black] (56.81,-45.4) -- (38.09,-54.8);
\fill [black] (38.09,-54.8) -- (39.03,-54.89) -- (38.58,-54);
\draw [black] (57.09,-45.81) -- (47.21,-54.39);
\fill [black] (47.21,-54.39) -- (48.14,-54.24) -- (47.49,-53.49);
\draw [black] (49.38,-33.92) -- (45.42,-42.68);
\fill [black] (45.42,-42.68) -- (46.21,-42.15) -- (45.3,-41.74);
\draw [black] (50.46,-34.08) -- (51.54,-42.52);
\fill [black] (51.54,-42.52) -- (51.94,-41.66) -- (50.95,-41.79);
\draw [black] (40.28,-22.3) -- (33.02,-30.6);
\fill [black] (33.02,-30.6) -- (33.92,-30.32) -- (33.17,-29.66);
\draw [black] (20.175,-10.803) arc (267.69007:-20.30993:1.5);
\fill [black] (21.64,-9.82) -- (22.17,-9.05) -- (21.18,-9);
\draw [black] (22.83,-13.57) -- (25.67,-19.03);
\fill [black] (25.67,-19.03) -- (25.75,-18.09) -- (24.86,-18.55);
\draw [black] (23.51,-12.98) -- (32.59,-19.62);
\fill [black] (32.59,-19.62) -- (32.24,-18.74) -- (31.65,-19.55);
\draw [black] (21.33,-13.72) -- (19.77,-18.88);
\fill [black] (19.77,-18.88) -- (20.48,-18.26) -- (19.53,-17.97);
\draw [black] (23.72,-12.63) -- (39.78,-19.97);
\fill [black] (39.78,-19.97) -- (39.26,-19.18) -- (38.85,-20.09);
\draw [black] (22.25,-13.77) -- (25.15,-30.13);
\fill [black] (25.15,-30.13) -- (25.5,-29.26) -- (24.52,-29.43);
\draw [black] (23.13,-13.37) -- (36.57,-30.53);
\fill [black] (36.57,-30.53) -- (36.47,-29.59) -- (35.68,-30.2);
\draw [black] (46.943,-21.172) arc (308.49659:20.49659:1.5);
\fill [black] (47.41,-19.47) -- (47.31,-18.54) -- (46.52,-19.16);
\draw [black] (47.23,-21.9) -- (33.37,-31);
\fill [black] (33.37,-31) -- (34.31,-30.98) -- (33.77,-30.14);
\draw [black] (49.13,-22.79) -- (49.97,-30.11);
\fill [black] (49.97,-30.11) -- (50.38,-29.26) -- (49.38,-29.38);
\draw [black] (38.23,-46.08) -- (44.47,-54.12);
\fill [black] (44.47,-54.12) -- (44.38,-53.18) -- (43.59,-53.8);
\draw [black] (35.51,-45.83) -- (25.89,-54.37);
\fill [black] (25.89,-54.37) -- (26.82,-54.21) -- (26.16,-53.47);
\draw [black] (11.022,-30.194) arc (224.61085:-63.38915:1.5);
\fill [black] (12.76,-30.48) -- (13.68,-30.28) -- (12.98,-29.56);
\draw [black] (12.55,-33.86) -- (23.45,-53.94);
\fill [black] (23.45,-53.94) -- (23.5,-53) -- (22.63,-53.48);
\draw [black] (44.156,-54.442) arc (258.56081:-29.43919:1.5);
\fill [black] (45.76,-53.71) -- (46.41,-53.02) -- (45.43,-52.83);
\draw [black] (42.655,-44.069) arc (285.24412:-2.75588:1.5);
\fill [black] (43.76,-42.69) -- (44.03,-41.79) -- (43.07,-42.05);
\draw [black] (42.85,-45.47) -- (26.15,-54.73);
\fill [black] (26.15,-54.73) -- (27.09,-54.78) -- (26.61,-53.9);
\draw [black] (44.69,-34.1) -- (45.61,-53.7);
\fill [black] (45.61,-53.7) -- (46.07,-52.88) -- (45.07,-52.93);
\draw [black] (43.94,-33.99) -- (36.96,-53.81);
\fill [black] (36.96,-53.81) -- (37.7,-53.22) -- (36.76,-52.89);
\draw [black] (43.3,-33.62) -- (25.7,-54.18);
\fill [black] (25.7,-54.18) -- (26.6,-53.9) -- (25.84,-53.25);
\draw [black] (10.31,-18.808) arc (207.43495:-80.56505:1.5);
\fill [black] (11.89,-19.6) -- (12.83,-19.67) -- (12.37,-18.78);
\draw [black] (9.28,-22.52) -- (4.62,-30.38);
\fill [black] (4.62,-30.38) -- (5.46,-29.95) -- (4.6,-29.44);
\draw [black] (9.87,-22.75) -- (5.53,-42.55);
\fill [black] (5.53,-42.55) -- (6.19,-41.87) -- (5.21,-41.66);
\draw [black] (10.58,-22.78) -- (13.42,-42.52);
\fill [black] (13.42,-42.52) -- (13.8,-41.66) -- (12.81,-41.8);
\draw [black] (64.113,-43.86) arc (279:-9:1.5);
\fill [black] (65.36,-42.61) -- (65.73,-41.74) -- (64.74,-41.9);
\draw [black] (64.13,-45.21) -- (38.17,-54.99);
\fill [black] (38.17,-54.99) -- (39.1,-55.18) -- (38.74,-54.24);
\draw [black] (54.878,-30.705) arc (253.27747:-34.72253:1.5);
\fill [black] (56.54,-30.12) -- (57.25,-29.5) -- (56.29,-29.21);
\draw [black] (55.62,-33.98) -- (52.48,-42.62);
\fill [black] (52.48,-42.62) -- (53.23,-42.04) -- (52.29,-41.7);
\draw [black] (54.69,-33.29) -- (26.01,-54.51);
\fill [black] (26.01,-54.51) -- (26.95,-54.44) -- (26.35,-53.63);
\draw [black] (62.61,-22.78) -- (65.69,-42.52);
\fill [black] (65.69,-42.52) -- (66.06,-41.66) -- (65.07,-41.81);
\draw [black] (62.14,-22.79) -- (61.56,-30.11);
\fill [black] (61.56,-30.11) -- (62.12,-29.35) -- (61.12,-29.27);
\draw [black] (61.99,-22.78) -- (58.91,-42.52);
\fill [black] (58.91,-42.52) -- (59.53,-41.81) -- (58.54,-41.66);
\draw [black] (32.33,-12.5) -- (12.17,-20.1);
\fill [black] (12.17,-20.1) -- (13.1,-20.28) -- (12.74,-19.35);
\draw [black] (33.318,-10.014) arc (234:-54:1.5);
\fill [black] (35.08,-10.01) -- (35.96,-9.66) -- (35.15,-9.07);
\draw [black] (33.318,-10.014) arc (234:-54:1.5);
\fill [black] (35.08,-10.01) -- (35.96,-9.66) -- (35.15,-9.07);
\draw [black] (36.14,-12.29) -- (67.86,-20.31);
\fill [black] (67.86,-20.31) -- (67.21,-19.63) -- (66.96,-20.6);
\draw [black] (32.49,-12.83) -- (20.91,-19.77);
\fill [black] (20.91,-19.77) -- (21.86,-19.79) -- (21.34,-18.93);
\draw [black] (35.11,-13.58) -- (43.69,-30.32);
\fill [black] (43.69,-30.32) -- (43.77,-29.38) -- (42.88,-29.84);
\draw [black] (36.1,-12.41) -- (60.4,-20.19);
\fill [black] (60.4,-20.19) -- (59.79,-19.47) -- (59.48,-20.42);
\draw [black] (34.55,-13.77) -- (37.45,-30.13);
\fill [black] (37.45,-30.13) -- (37.8,-29.26) -- (36.82,-29.43);
\draw [black] (67.01,-34.09) -- (66.19,-42.51);
\fill [black] (66.19,-42.51) -- (66.77,-41.76) -- (65.77,-41.66);
\draw [black] (65.932,-30.563) arc (247.24591:-40.75409:1.5);
\fill [black] (67.65,-30.16) -- (68.42,-29.61) -- (67.5,-29.23);
\draw [black] (66.06,-33.74) -- (59.74,-42.86);
\fill [black] (59.74,-42.86) -- (60.61,-42.48) -- (59.78,-41.91);
\draw [black] (55.23,-22.79) -- (56.07,-30.11);
\fill [black] (56.07,-30.11) -- (56.48,-29.26) -- (55.48,-29.38);
\draw [black] (53.821,-19.194) arc (244.00798:-43.99202:1.5);
\fill [black] (55.56,-18.89) -- (56.36,-18.39) -- (55.46,-17.95);
\draw [black] (55.99,-22.54) -- (60.41,-30.36);
\fill [black] (60.41,-30.36) -- (60.46,-29.42) -- (59.59,-29.91);
\draw [black] (54.22,-22.64) -- (50.98,-30.26);
\fill [black] (50.98,-30.26) -- (51.75,-29.72) -- (50.83,-29.33);
\draw [black] (53.79,-22.39) -- (38.21,-42.91);
\fill [black] (38.21,-42.91) -- (39.09,-42.57) -- (38.3,-41.97);
\draw [black] (50.02,-13.46) -- (53.88,-19.14);
\fill [black] (53.88,-19.14) -- (53.84,-18.2) -- (53.02,-18.76);
\draw [black] (48.9,-13.8) -- (48.9,-18.8);
\fill [black] (48.9,-18.8) -- (49.4,-18) -- (48.4,-18);
\draw [black] (48.018,-10.014) arc (234:-54:1.5);
\fill [black] (49.78,-10.01) -- (50.66,-9.66) -- (49.85,-9.07);
\draw [black] (49.03,-13.8) -- (50.07,-30.1);
\fill [black] (50.07,-30.1) -- (50.52,-29.27) -- (49.52,-29.34);
\draw [black] (47.64,-13.35) -- (42.86,-19.25);
\fill [black] (42.86,-19.25) -- (43.75,-18.94) -- (42.98,-18.31);
\draw [black] (50.56,-12.92) -- (60.64,-19.68);
\fill [black] (60.64,-19.68) -- (60.25,-18.82) -- (59.7,-19.65);
\draw [black] (47.94,-13.55) -- (38.76,-30.35);
\fill [black] (38.76,-30.35) -- (39.58,-29.88) -- (38.7,-29.4);
\draw [black] (60.242,-30.479) arc (243.27703:-44.72297:1.5);
\fill [black] (61.98,-30.19) -- (62.79,-29.71) -- (61.9,-29.26);
\draw [black] (60.29,-33.77) -- (46.81,-54.03);
\fill [black] (46.81,-54.03) -- (47.67,-53.65) -- (46.83,-53.09);
\draw [black] (69.48,-22.77) -- (66.32,-42.53);
\fill [black] (66.32,-42.53) -- (66.94,-41.81) -- (65.95,-41.66);
\draw [black] (69.35,-22.75) -- (67.65,-30.15);
\fill [black] (67.65,-30.15) -- (68.32,-29.48) -- (67.34,-29.26);
\draw [black] (68.286,-19.505) arc (257.19859:-30.80141:1.5);
\fill [black] (69.91,-18.81) -- (70.57,-18.14) -- (69.6,-17.92);
\draw [black] (68.286,-19.505) arc (257.19859:-30.80141:1.5);
\fill [black] (69.91,-18.81) -- (70.57,-18.14) -- (69.6,-17.92);
\draw [black] (68.95,-22.61) -- (59.45,-42.69);
\fill [black] (59.45,-42.69) -- (60.25,-42.18) -- (59.34,-41.75);
\draw [black] (15.49,-45.39) -- (34.51,-54.81);
\fill [black] (34.51,-54.81) -- (34.01,-54.01) -- (33.57,-54.9);
\draw [black] (15.08,-45.95) -- (23.02,-54.25);
\fill [black] (23.02,-54.25) -- (22.83,-53.33) -- (22.1,-54.02);
\draw [black] (1.796,-31.256) arc (272.65981:-15.34019:1.5);
\fill [black] (3.17,-30.15) -- (3.63,-29.33) -- (2.64,-29.38);
\draw [black] (3.84,-34.09) -- (4.86,-42.51);
\fill [black] (4.86,-42.51) -- (5.26,-41.66) -- (4.27,-41.78);
\draw [black] (4.86,-33.65) -- (12.44,-42.95);
\fill [black] (12.44,-42.95) -- (12.32,-42.01) -- (11.54,-42.64);
\draw [black] (34.309,-55.78) arc (300.03751:12.03751:1.5);
\fill [black] (35.03,-54.17) -- (35.06,-53.23) -- (34.19,-53.73);
\draw [black] (49.835,-44.173) arc (288.2757:0.2757:1.5);
\fill [black] (50.86,-42.74) -- (51.09,-41.82) -- (50.14,-42.14);
\draw [black] (49.95,-45.26) -- (26.25,-54.94);
\fill [black] (26.25,-54.94) -- (27.18,-55.1) -- (26.8,-54.18);
\draw [black] (22.544,-54.977) arc (276.44623:-11.55377:1.5);
\fill [black] (23.84,-53.79) -- (24.25,-52.94) -- (23.26,-53.05);
\draw [black] (5.593,-42.57) arc (193.39871:-94.60129:1.5);
\fill [black] (6.93,-43.72) -- (7.83,-44.02) -- (7.59,-43.05);
\draw [black] (6.98,-45.18) -- (34.42,-55.02);
\fill [black] (34.42,-55.02) -- (33.83,-54.28) -- (33.5,-55.22);
\draw [black] (24.94,-34.02) -- (22.46,-42.58);
\fill [black] (22.46,-42.58) -- (23.16,-41.95) -- (22.2,-41.67);
\draw [black] (26.8,-33.62) -- (44.4,-54.18);
\fill [black] (44.4,-54.18) -- (44.26,-53.25) -- (43.5,-53.9);
\draw [black] (26.14,-33.99) -- (29.06,-42.61);
\fill [black] (29.06,-42.61) -- (29.28,-41.69) -- (28.33,-42.01);
\draw [black] (25,-22) -- (13.2,-30.9);
\fill [black] (13.2,-30.9) -- (14.14,-30.81) -- (13.54,-30.02);
\draw [black] (26.086,-18.875) arc (222.69007:-65.30993:1.5);
\fill [black] (27.81,-19.22) -- (28.74,-19.05) -- (28.06,-18.31);
\draw [black] (25.43,-22.43) -- (19.67,-30.47);
\fill [black] (19.67,-30.47) -- (20.54,-30.12) -- (19.72,-29.53);
\draw [black] (27.4,-22.63) -- (36.2,-42.67);
\fill [black] (36.2,-42.67) -- (36.33,-41.74) -- (35.42,-42.14);
\draw [black] (26.41,-22.79) -- (25.69,-30.11);
\fill [black] (25.69,-30.11) -- (26.27,-29.36) -- (25.27,-29.26);
\draw [black] (19.909,-44.558) arc (299.3989:11.3989:1.5);
\fill [black] (20.64,-42.95) -- (20.69,-42.01) -- (19.81,-42.5);
\draw [black] (23.71,-45.35) -- (43.89,-54.85);
\fill [black] (43.89,-54.85) -- (43.38,-54.06) -- (42.95,-54.96);
\draw [black] (31.252,-30.159) arc (220.7378:-67.2622:1.5);
\fill [black] (32.97,-30.56) -- (33.9,-30.42) -- (33.25,-29.66);
\draw [black] (33.14,-33.49) -- (43.16,-43.11);
\fill [black] (43.16,-43.11) -- (42.93,-42.2) -- (42.23,-42.92);
\draw [black] (31.38,-34.07) -- (30.02,-42.53);
\fill [black] (30.02,-42.53) -- (30.64,-41.82) -- (29.65,-41.66);
\draw [black] (32.49,-33.94) -- (36.21,-42.66);
\fill [black] (36.21,-42.66) -- (36.36,-41.73) -- (35.44,-42.12);
\draw [black] (33.318,-19.014) arc (234:-54:1.5);
\fill [black] (35.08,-19.01) -- (35.96,-18.66) -- (35.15,-18.07);
\draw [black] (33.77,-22.75) -- (32.13,-30.15);
\fill [black] (32.13,-30.15) -- (32.79,-29.47) -- (31.82,-29.26);
\draw [black] (32.98,-22.38) -- (26.72,-30.52);
\fill [black] (26.72,-30.52) -- (27.6,-30.19) -- (26.81,-29.58);
\draw [black] (27.734,-44.819) arc (306.95658:18.95658:1.5);
\fill [black] (28.25,-43.13) -- (28.17,-42.19) -- (27.37,-42.79);
\draw [black] (31.34,-45.65) -- (44.06,-54.55);
\fill [black] (44.06,-54.55) -- (43.69,-53.68) -- (43.12,-54.5);
\draw [black] (17.618,-30.314) arc (234:-54:1.5);
\fill [black] (19.38,-30.31) -- (20.26,-29.96) -- (19.45,-29.37);
\draw [black] (20.01,-33.41) -- (44.19,-54.39);
\fill [black] (44.19,-54.39) -- (43.91,-53.49) -- (43.26,-54.24);
\draw [black] (19.03,-34.03) -- (21.37,-42.57);
\fill [black] (21.37,-42.57) -- (21.64,-41.67) -- (20.68,-41.93);
\draw [black] (48.72,-33.45) -- (25.88,-54.35);
\fill [black] (25.88,-54.35) -- (26.8,-54.18) -- (26.13,-53.44);
\draw [black] (37.1,-4.07) -- (34.9,-9.93);
\fill [black] (34.9,-9.93) -- (35.65,-9.35) -- (34.71,-9);
\draw [black] (36.09,-3.23) -- (23.61,-10.77);
\fill [black] (23.61,-10.77) -- (24.56,-10.78) -- (24.04,-9.92);
\draw [black] (39.31,-3.51) -- (47.39,-10.49);
\fill [black] (47.39,-10.49) -- (47.11,-9.59) -- (46.46,-10.35);
\draw [black] (39.39,-3.41) -- (60.71,-19.59);
\fill [black] (60.71,-19.59) -- (60.37,-18.71) -- (59.77,-19.51);
\draw [black] (37.8,-4.2) -- (37.8,-30.1);
\fill [black] (37.8,-30.1) -- (38.3,-29.3) -- (37.3,-29.3);
\draw [black] (38.2,-4.16) -- (41.2,-18.84);
\fill [black] (41.2,-18.84) -- (41.53,-17.96) -- (40.55,-18.16);
\draw [black] (36.39,-3.61) -- (20.61,-19.39);
\fill [black] (20.61,-19.39) -- (21.53,-19.17) -- (20.83,-18.47);
\draw [black] (63.1,-22.63) -- (66.4,-30.27);
\fill [black] (66.4,-30.27) -- (66.54,-29.33) -- (65.63,-29.73);

\draw [black] (37.8,-32.1) circle (2);
\draw (37.8,-32.1) node {$111$};
\draw [black] (19.2,-20.8) circle (2);
\draw (19.2,-20.8) node {$112$};
\draw [black] (58.6,-44.5) circle (2);
\draw (58.6,-44.5) node {$113$};
\draw [black] (50.2,-32.1) circle (2);
\draw (50.2,-32.1) node {$114$};
\draw [black] (41.6,-20.8) circle (2);
\draw (41.6,-20.8) node {$121$};
\draw [black] (21.9,-11.8) circle (2);
\draw (21.9,-11.8) node {$122$};
\draw [black, fill=white] (48.9,-20.8) circle (2);
\draw (48.9,-20.8) node {$124$};
\draw [black, fill=white] (37,-44.5) circle (2);
\draw (37,-44.5) node {$131$};
\draw [black, fill=white] (11.6,-32.1) circle (2);
\draw (11.6,-32.1) node {$132$};
\draw [black, fill=white] (45.7,-55.7) circle (2);
\draw (45.7,-55.7) node {$133$};
\draw [black, fill=white] (44.6,-44.5) circle (2);
\draw (44.6,-44.5) node {$134$};
\draw [black, fill=white] (44.6,-32.1) circle (2);
\draw (44.6,-32.1) node {$141$};
\draw [black] (10.3,-20.8) circle (2);
\draw (10.3,-20.8) node {$142$};
\draw [black] (66,-44.5) circle (2);
\draw (66,-44.5) node {$143$};
\draw [black] (56.3,-32.1) circle (2);
\draw (56.3,-32.1) node {$144$};
\draw [black, fill=white] (62.3,-20.8) circle (2);
\draw (62.3,-20.8) node {$211$};
\draw [black, fill=white] (34.2,-11.8) circle (2);
\draw (34.2,-11.8) node {$212$};
\draw [black, fill=white] (67.2,-32.1) circle (2);
\draw (67.2,-32.1) node {$213$};
\draw [black, fill=white] (55,-20.8) circle (2);
\draw (55,-20.8) node {$214$};
\draw [black, fill=white] (48.9,-11.8) circle (2);
\draw (48.9,-11.8) node {$221$};
\draw [black, fill=white] (61.4,-32.1) circle (2);
\draw (61.4,-32.1) node {$231$};
\draw [black, fill=white] (69.8,-20.8) circle (2);
\draw (69.8,-20.8) node {$241$};
\draw [black, fill=white] (13.7,-44.5) circle (2);
\draw (13.7,-44.5) node {$311$};
\draw [black, fill=white] (3.6,-32.1) circle (2);
\draw (3.6,-32.1) node {$312$};
\draw [black, fill=white] (36.3,-55.7) circle (2);
\draw (36.3,-55.7) node {$313$};
\draw [black, fill=white] (51.8,-44.5) circle (2);
\draw (51.8,-44.5) node {$314$};
\draw [black, fill=white] (24.4,-55.7) circle (2);
\draw (24.4,-55.7) node {$331$};
\draw [black, fill=white] (5.1,-44.5) circle (2);
\draw (5.1,-44.5) node {$341$};
\draw [black, fill=white] (25.5,-32.1) circle (2);
\draw (25.5,-32.1) node {$411$};
\draw [black, fill=white] (26.6,-20.8) circle (2);
\draw (26.6,-20.8) node {$412$};
\draw [black, fill=white] (21.9,-44.5) circle (2);
\draw (21.9,-44.5) node {$413$};
\draw [black, fill=white] (31.7,-32.1) circle (2);
\draw (31.7,-32.1) node {$414$};
\draw [black, fill=white] (34.2,-20.8) circle (2);
\draw (34.2,-20.8) node {$421$};
\draw [black, fill=white] (29.7,-44.5) circle (2);
\draw (29.7,-44.5) node {$431$};
\draw [black, fill=white] (18.5,-32.1) circle (2);
\draw (18.5,-32.1) node {$441$};
\draw [black, fill=white] (37.8,-2.2) circle (2);
\draw (37.8,-2.2) node {$S$};
\end{tikzpicture}\\
    \caption{The automaton $A_3$.}
    \label{fig:a3}
\end{figure}
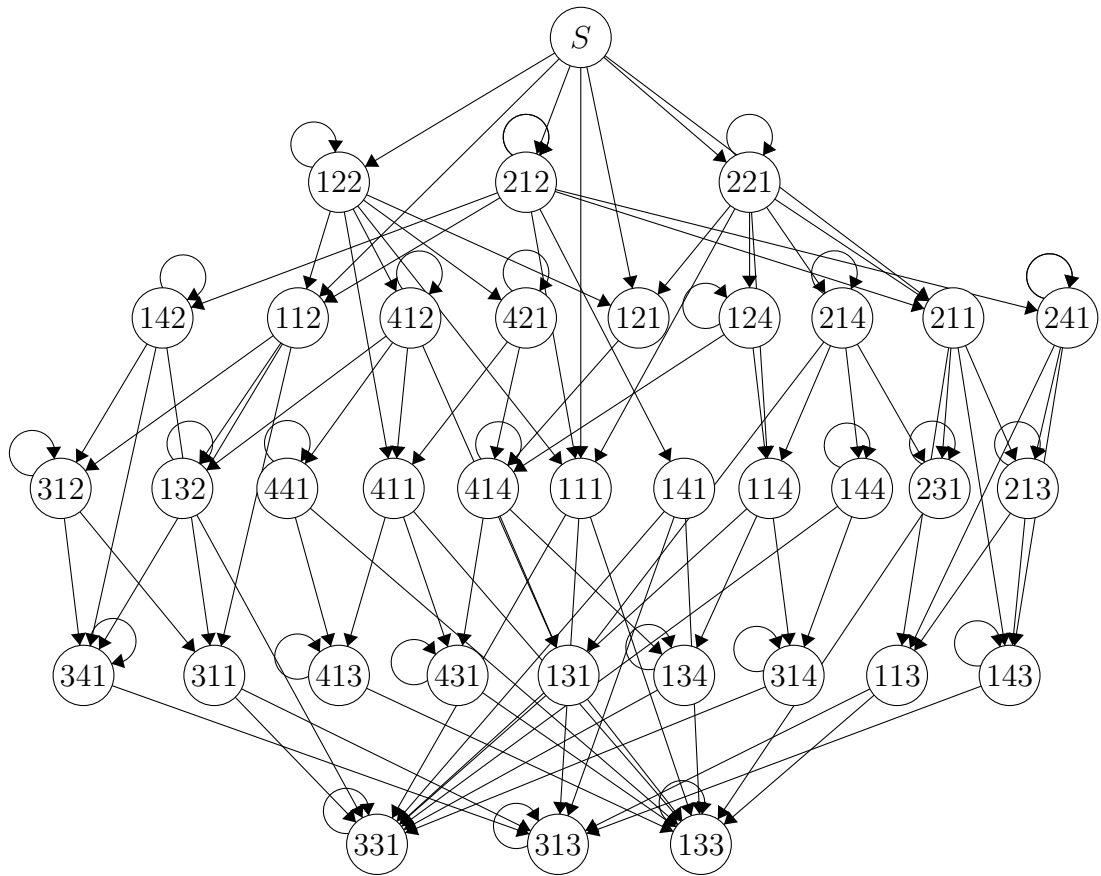
\subsection{Depth of States}

I have drawn the automata like this to motivate the following definition.
\begin{definition}
The \textbf{depth} of a state, $s$, in $A_r$, is equal to $r + t_3 - t_2$,  where $t_3$ is the number of 3s that can be found in the rowstates of $s$, and $t_2$ is the number of 2s that can be found in the rowstates of $s$.  Let $d(s)$ denote the depth.
\end{definition}
\begin{lemma}
For a fixed number of rows, $r$, any transition in $A_r$ must either be a self-loop or increase depth. Additionally, a self-loop emerges from a state if and only if the corresponding column has a single 1. 
\end{lemma}
\begin{proof}
Consider an arbitrary transition. The two states in the transition, let's call them $s$ and $s'$, fully specify the contents of two consecutive columns of the matrix, let's call them $m$ and $m+1$ respectively. If a row in $s$ is in rowstate 3, it must be in rowstate 3 in $s'$ as well. Similarly if a row in $s'$ is in rowstate 2, it must be in rowstate 2 in $s$ as well. Therefore the contribution to depth from a single row cannot decrease. The total depth is computed as a sum of contributions from each row, so it follows that the total depth cannot decrease. To show that if the depth remains the same, we must have a self transition, we first note that each individual row's contribution in a depth preserving transition must not increase depth. Otherwise, if such a row did increase its contribution, then we can apply the above reasoning to the other rows, and conclude that the total depth must increase. Therefore all 2's in $s$ must remain 2's, and all 3's in $s'$ must have come from 3's in $s$. 

We know that both columns must contain a 1 somewhere. Say the 1 in column $m$ is in row $i$. It cannot be the case that the only 1 in column $m+1$ is in row $i$, otherwise we would have a self-loop. Therefore there exists a row $j \neq i$ so that column $m+1$ contains a 1 in row $j$. Let's assume $j < i$, and note that the other case will be covered by symmetry. The clockwise pinwheel of index $(j,m)$ indicates that row $j+1$ must end up in rowstate 3. Since the depth is preserved it must have started in rowstate 3 as well. Now consider the clockwise pinwheel of index $(j+1,m)$. Since we cannot have row $j+1$ starting in rowstate 2,  we require row $j+2$ to end up in rowstate 3 as well. This pattern continues all the way down to row $i$ where we start in rowstate 1 and end in rowstate 3, which necessarily increases the depth.

In fact we have shown something stronger: whenever there exists 1s in consecutive columns that are not in the same row, the corresponding transition must increase depth. Thus a self-loop can only emerge from states that contain a single 1. It is straightforward to check that if a state has a single 1, then the corresponding self-loop does not break any rules, and is present.  \qedhere

\end{proof}

\subsection{Counting Baxter Matrices
}
Suppose we fix the number of rows, $r$, and want to count the number of Baxter matrices with $k$ columns. They are in bijection with paths of length $k$ on our state transition graph. If we fix $r$ and ignore the self-loops, the lemma shows that there are only finitely many possible paths. Thus we can classify all Baxter matrices with $r$ rows into finitely many classes according to the paths they take with all the self-loops removed. 

How many $k$ column matrices are there of a particular class? Let $P$ denote a path of length $l$ in the automaton that starts from $S$ and avoids loops, but goes through $q$ states that contain loops. Let $N(P,k)$ be the number of paths of length $k$ arising from $P$.  To compute $N(P,k)$, we just have to choose how many self-loops to put at each node where it is possible to put self-loops, so that the total path length is $k$. This is the number of ways to choose $q$ natural numbers that sum to $k-l$. Using stars and bars we get that if $k\geq l$: $$N(P,k) = \binom{k-l+q-1}{q-1} $$ which is a polynomial in $k$ of degree $q-1$. Note that the polynomial is also correct for $l+1-q \leq k < l$, since it correctly outputs 0 in these cases.

Notice that the minimum possible depth of a state is $0$ at the start state, and the maximum possible depth is $2r-1$. Thus $l \leq  2r-1$, and therefore $q \leq 2r-1$ as well. Thus for $k \geq 2r-1$, $N(P,k)$ coincides with a polynomial in $k$ of degree at most $2r-2$. 

We can also show that the bound is tight, and that there exists a path $P^*$ with corresponding polynomial having degree exactly $2r-2$. To do this it suffices to force $P^*$ to go through exactly $2r-1$ states with self-loops, i.e. set $q = 2r-1$. Let $M^*$ be the $r \times (2r-1)$ Matrix that is all 0s except for the following locations: $M^*[i,i] = 1$ for $1 \leq i \leq r$, and $M^*[i,2r-i] = 1$ for $1 \leq i \leq r$. Thus $M^*$ has two diagonal stripes of 1s that intersect in the middle. $M^*$ has no self-loops since it has no repeated columns, but since each column has a single 1, all of the states that is passes through have legal self-loops. It is straightforward to check that $M^*$ is indeed a Baxter matrix, and then we can set $P^*$ to be the path in $A_r$ that corresponds to $M^*$.

Now for each possible path $P$ without self-loops, $N(P,k)$ is eventually a polynomial in $k$ of degree at most $2r-2$. The total number of matrices with $r$ rows and $k$ columns will be the sum over all these polynomials, which will be a polynomial of degree exactly $2r-2$. The only small caveat is that for small values of $k$ we must take the maximum of each polynomial and 0, so some paths cannot contribute negatively. Summarizing:
\begin{theorem}

Let $a_{r,l,q}$ be the number of paths of length $l$ in $A_r$ that start from $S$, do not use self-loop transitions, but pass by $q$ states carrying a loop. 
Let $P_r(x)$ be the polynomial in $x$ defined by
$$P_r(x) = \sum_{l,q \geq 1} a_{r,l,q} \binom{x-l+q-1}{q-1}$$ Then $P_r(x)$ has degree $2r-2$. 

Let $N_{r,k}$ be the number of Baxter matrices of size $r \times k$. Then, for $k \geq 2r-1$ we have $N_{r,k} = P_r(k)$. 
\end{theorem}
This shows that for a fixed number of rows, $r$, the number of Baxter matrices with $r$ rows and $k$ columns eventually satisfies a polynomial in $k$ of degree $2r-2$. Computing the polynomial for a fixed $r$ is straightforward once the transition graph has been constructed. The author has Maple code that constructs the graph and computes the corresponding polynomial using the above method. The code completes instantly for $r\leq 5$ and within a couple minutes for $r=6$. The polynomial for $r=2$ also appears in Knuth's paper, who found it using a combinatorial approach.

\begin{tabular}{lll}

rows & formula & works for  \\ \hline 
$2$ & $k^2+3k-4$ & $k \geq 2$   \\ 
$3$ & $(1/3)k^4 + 3k^3 - (16/3)k^2 + 2k + 3$& $k \geq 3$\\ 
$4$ &$(1/18)k^6 + (21/20)k^5 - (5/18)k^4 - (151/12)k^3 $&  $k \geq 4$ \\ 
& $+ (443/9)k^2 - (1012/15)k + 28 $&\\
$5$ & $(23/4032)k^8 + (937/5040)k^7 + (853/1440)k^6 -(2671/360)k^5  $& $k \geq 5$ \\
& $+(15697/576)k^4-(341/720 )k^3-(1274363/5040 )k^2$&\\
&$+(98659/140) k-643$&\\
$6$ & $(361/907200)k^{10} + (403/20160)k^9 + (5177/30240)k^8+\ldots$& $k \geq 6$
\end{tabular}

The reader will notice that the polynomials are correct for $k \geq r$, which is a much better bound than $k \geq 2r-1$. Does this pattern continue for larger r? The answer is yes, and we can prove it using Lemma \ref{extra} from the next section. Recall that for any path $P$ in $A_r$ without self-loops, the polynomial corresponding to $P$ is correct for $k \geq l + 1 -q$, where $l$ is the length of $P$, and $q$ is the number of states with self-loops that $P$ passes through. Thus if we can show that for any $P$, $l +1-q \leq r$, we can show that $N_{r,k} = P_r(k)$, for $k \geq r$. To compute the maximum of $l +1-q$, we note that $l-q$ is equal to the number of states without self-loops that $P$ passes through. Using the terminology of the next section, $l-q$ is equal the number of states with extra 1s that $P$ passes through. By Lemma \ref{extra}, $P$ cannot pass through two states with extra 1s on consecutive depth levels. We know $l \leq 2r-1$, and since the first possible depth a state with extra 1s can occur on is depth 2, we get that $l-q \leq r-1$. Thus $l+1-q \leq r$.

\section{Resolving one of Knuth's conjectures}
Donald Knuth conjectured that the number of 1s in any $r \times k$ Baxter matrix is fewer than $(r+k)$. We know from the definition of Baxter matrices that each column must contain at least one 1. Then we can refer to any 1 that is not the topmost in its column as an extra 1.

\begin{theorem}
The number of extra 1s in a Baxter matrix with $r$ rows is less than $r$.
\end{theorem}
To prove this, we will use the following lemma.
\begin{lemma}\label{extra}
The total number of extra 1s that appear in two consecutive columns is at most the change in depth of the corresponding state transition in $A_r$.
\end{lemma}

For now let's assume the lemma is true. Suppose $M$ is some Baxter matrix with $r$ rows and $k$ columns. Let $p$ be its corresponding path in $A_r$, and let $T$ be the set of transitions in $p$. Let $t^*$ be the final state in $p$. If we use the fact that the start state does not contain any extra 1s and assume that $t^*$ does not contain any extra 1s, we get that
\begin{align}
     (\text{\# of extra 1s in } M) & = \frac{1}{2}\left(\sum_{\tau \in T}(\text{\# of extra 1s in the columns associated with } \tau) \right)
\end{align}

Suppose $t^*$ does contain extra 1s. Let $t'$ be the state obtained by replacing all the extra 1s in $t^*$ with rowstate 3. We now modify $p$ to contain an extra transition from $t^*$ to $t'$, regardless of whether this transition exists in $A_r$. Note that $t'$ does not have any extra 1s by construction. This extra transition satisfies Lemma \ref{extra}, because for each extra 1 in $t^*$, the depth has increased by 1 in that row. Now we can apply Lemma \ref{extra} to the (potentially modified) $p$.

\begin{align}
    (\text{\# of extra 1s in } M) & = \frac{1}{2}\left(\sum_{\tau \in T}(\text{\# of extra 1s in the columns associated with } \tau) \right) \\
    & \leq \frac{1}{2}\left(\sum_{\tau \in T}(\text{depth increase of } \tau) \right) \\
    & \leq \frac{1}{2}(2r-1) \\
    & < r
\end{align}
Now for a proof of Lemma \ref{extra}:
\begin{proof}
Consider 2 consecutive columns, column $m$ and column $m+1$. We will denote them by $c_0$ and $c_1$ for short.    Let $i$ be the minimum row which contains a 1 in $c_0$, $j$ be the minimum row that contains a 1 in $c_1$,  $k$ be the maximum row which contains a 1 in $c_0$, and $l$ be the maximum row that contains a 1 in $c_1$. We can assume $i \leq j$, for if not the proof is symmetric. We will consider 3 cases. 

\textbf{Case 1:} $i\leq k \leq j \leq l$.  We will show that each extra 1 in $c_0$ forces the depth to increase by at least one. If $i=k$ then there are no extra 1s so there is nothing to show. So we have $i < k \leq j$. Let's look at the counterclockwise pinwheel of index $(k-1,m)$. The only way to satisfy it is if row $k-1$ ends up in rowstate 3. We can then look at the counterclockwise pinwheel of index $(k-2,m)$, and conclude that row $k-2$ must end up in rowstate 3 as well (the left segment cannot be used otherwise that row would be all 0s). We get that each row from $k-1$ up to $i$ ends up in rowstate 3, so we have an increase in depth for each 1 in $c_0$ from rows $i$ to $k-1$, which is an increase in depth by 1 for each extra 1 in $c_0$. Now look at the counterclockwise pinwheel of index $(j,m)$, and recall that $j \geq k$. It forces row $j+1$ to start in rowstate 2. Applying the same process, all rows from row $j+1$ to $l$ must start in rowstate 2. This gives an increase in depth for each 1 in $c_1$ between row $j+1$ and $l$.

If we are not in case 1, then it must be the case that $k > j$.

\textbf{Case 2:} $i \leq j = l < k$. We will show that for each 1 that is not on row $j$, there must be an increase in depth. If there is a 1 above row $j$, then the counterclockwise pinwheel of index $(j-1,m)$ can only be satisfied if row $j-1$ ends up in rowstate 3. Continuing to look at the counterclockwise pinwheels upward, all rows must end up in rowstate 3 from row $j-1$ up to row $i$. Thus all 1s in those rows cause there to be transitions from rowstate 1 into rowstate 3, causing an increase in depth. Similarly the clockwise pinwheel of index $(j,m)$ requires row $j+1$ to end up in rowstate 3, and all rows below it down to row $k$ as well. Thus each 1 below row $j$ also causes an increase in depth.

\textbf{Case 3:} $i \leq j < l$ and $k> j$. It is not hard to show that in this case $i < j$ and there are no 1s in the first column on rows $j$ through $l$. Let $i'$ be the maximum row that contains a 1 that is less than $j$, and $k'$ be the minimum row that contains a 1 that is greater than $l$. We will show that each 1 except for the 1s in rows $i'$ and $k'$ cause an increase in depth. By the same reasoning that we applied in case 1, there must be an increase in depth for each 1 above row $i'$ and for each 1 below row $k'$. If there are at least two 1s in the $c_1$, then each of those 1s must have transitioned from rowstate 2 in $c_0$, causing an increase in depth. If there is a single 1, then we can look at the clockwise pinwheel below it to conclude it either transitioned from rowstate 2 or the next row ends up in rowstate 3. In the first case we are done and in the second case we can progress the rowstate 3's down until we get that row $k'$ must end up in rowstate 3. \qedhere

\end{proof}
Note that the number of extra 1s in a Baxter matrix is completely determined by its path in $A_r$. Also Lemma \ref{extra} shows that a self-loop never can produce extra 1s. Therefore the self-loops in a path have no effect on the number of extra 1s. If we want to compute the number of $r \times k$ Baxter matrices with $t$ 1s for some $k \leq t < k+r$, we can use the same method that we used to count them before, except only add the polynomials for paths without self-loops that add the appropriate amount of extra 1s. Here are some results:

$r=3$, correct for $k \geq 3$

\begin{tabular}{lll}
extra 1s & total weight & formula  \\ \hline 
0 & $k$ & $(1/3)k^4 - k^3 + (2/3)k^2$ \\ 
1 & $k+1$& $4k^3 - 12k^2 + 15k - 8$\\ 
2 &$k+2$&$6k^2 - 13k + 11$
\end{tabular}
\vskip .5cm

$r=4$, correct for $k \geq 4$

\begin{tabular}{lll}

1s & weight & formula  \\ \hline 
0 & $k$ & $(1/18)k^6 - (3/10)k^5 + (2/9)k^4 + (3/2)k^3 - (77/18)k^2 + (24/5)k - 2$ \\ 
1 & $k+1$& $(27/20)k^5 - (47/6)k^4 + (235/12)k^3 - (157/6)k^2 + (226/15)k$\\ 
2 &$k+2$&$(22/3)k^4 - (121/3)k^3 + (335/3)k^2 - (500/3)k + 106$\\
3 &$k+3$&$(20/3)k^3 - 32k^2 + (238/3)k - 76$
\end{tabular}
\vskip .5cm
There is a github repository for the maple code, and the code is open source under the MIT license. 
\url{https://github.com/DarthCalculus/BaxterMatrices}


\end{document}